\documentclass[12pt]{amsart}
\usepackage{latexsym,amsmath,amssymb,epsfig,amsthm}
\usepackage{graphicx}
\usepackage{tikz}
\usetikzlibrary{calc}
\usepackage[enableskew,vcentermath]{youngtab}
\usepackage{hyperref}
\hypersetup{colorlinks=false}
\input epsf
\newtheorem{theorem}{Theorem}[section]
\newtheorem{proposition}[theorem]{Proposition}

\newtheorem{conjecture}[theorem]{Conjecture}

\newtheorem{lemma}[theorem]{Lemma}
\newtheorem{remark}[theorem]{Remark}

\newcommand{\des}{{\rm des}}

\newcommand{\Ehr}{{\rm Ehr}}

\newcommand{\exc}{{\rm exc}}

\newcommand{\unlucky}{{\rm unlucky}}

\newcommand{\dD}{{\mathcal D}}

\newcommand{\fF}{{\mathcal F}}

\newcommand{\qQ}{{\mathcal Q}}
\newcommand{\sS}{{\mathcal S}}
\newcommand{\wW}{{\mathcal W}}

\newcommand{\RR}{{\mathbb R}}

\newcommand{\NN}{{\mathbb N}}

\renewcommand{\to}{\rightarrow}

\newcommand{\sm}{{\smallsetminus}}
\begin{document}
\title[Lattice point enumeration of some arbor polytopes]
{Lattice point enumeration of some arbor polytopes}

\author{Christos~A.~Athanasiadis}
\address{Department of Mathematics\\
National and Kapodistrian University of Athens\\
Panepistimioupolis\\ 15784 Athens, Greece}
\email{caath@math.uoa.gr}

\author{Qiqi~Xiao}
\address{School of Mathematical Sciences\\
Dalian University of Technology\\
Dalian, Liaoning 116024, PR China}
\email{xiaoqqcs@hotmail.com}

\author{Xue~Yan}
\address{School of Mathematical Sciences\\
Qufu Normal University\\
Qufu 273165, PR China}
\email{yanxue@qfnu.edu.cn}

\date{March 12, 2026}

\begin{abstract}
The $n$-dimensional lattice polytopes 
$\mathcal{Q}_{n,k}$ 
obtained by intersecting the $n$th dilate of the 
standard $n$-dimensional simplex in $\mathbb{R}^n$ 
with the half-spaces $x_i \le 1$ for $1 \le i \le k$ 
form an interesting special case of Chapoton's arbor 
polytopes. They interpolate between the $n$th 
dilate of the standard $n$-dimensional simplex and 
the standard $n$-dimensional cube in $\mathbb{R}^n$. 
This paper provides an 
explicit combinatorial interpretation of the 
$h^\ast$-polynomial of $\mathcal{Q}_{n,k}$, as 
the ascent enumerator of certain words, and partly 
confirms some of Chapoton's conjectures on the 
lattice point enumeration of arbor polytopes in 
this special case. More specifically, the Ehrhart
polynomial of $\mathcal{Q}_{n,k}$ is shown to be 
magic positive, by means of a new combinatorial 
parking model for cars, and the real-rootedness 
of its $h^\ast$-polynomial is deduced. The 
polynomial whose coefficients count the lattice 
points of $\mathcal{Q}_{n,k}$ by the number of 
their nonzero coordinates is shown to be 
gamma-positive and a combinatorial 
interpretation of the $h^\ast$-polynomial of 
any arbor polytope is conjectured. 
\end{abstract}

\maketitle

\section{Introduction}
\label{sec:intro}

Given a positive integer $n$ and $k \in \{0,
1,\dots,n\}$, we will be primarily interested in
the polytope $\qQ_{n,k}$ in $\RR^n$ defined by the
inequalities
\begin{equation} \label{eq:defQnk}
\begin{tabular}{l}
$x_i \ge 0$ \textrm{for} $1 \le i \le n$, \\
[1mm] $x_i \le 1$ \textrm{for} $1 \le i \le k$,
and \\ [1mm] $x_1 + x_2 + \cdots + x_n \le n$,
\end{tabular}
\end{equation}
where $x_1, x_2,\dots,x_n$ are the standard
coordinates of $\RR^n$. This $n$-dimensional 
lattice polytope specializes to the $n$th dilate 
of the standard $n$-dimensional simplex and to 
the standard $n$-dimensional cube in $\RR^n$ 
for $k=0$ and $k=n$, respectively.

The polytopes $\qQ_{n,k}$ belong to the class of
arbor polytopes, introduced and studied by 
Chapoton~\cite{Cha25+}. Let us recall their 
definition. An \emph{arbor} on the ground set $[n] 
:= \{1, 2,\dots,n\}$ is a rooted tree with vertices 
the blocks of a partition of $[n]$. We call $n$ 
the \emph{size} of the arbor. Given such an arbor 
$\tau$, the polytope $\qQ_\tau$ in $\RR^n$ is 
defined by the inequalities $x_i \ge 0$ for $i \in 
[n]$ and 
\[ \sum_{i \in \dD(v)} x_i \le |\dD(v)| \]
for all vertices $v$ of $\tau$, where $\dD(v)$ is
the union of all vertices which are descendants of
$v$ in $\tau$ (including $v$ itself). Thus, $x_1 
+ x_2 + \cdots + x_n \le n$ is the inequality 
which corresponds to the root of $\tau$. It was 
verified in \cite[Section~1]{Cha25+} that 
$\qQ_\tau$ is an $n$-dimensional lattice polytope 
(all vertices of $\qQ_\tau$ belong to $\NN^n$) 
which is simple (every vertex is an endpoint of 
exactly $n$ edges of $\qQ_\tau$).

Chapoton~\cite{Cha25+} formulated a number of 
intriguing conjectures on the lattice point 
enumeration of arbor polytopes. The first of 
these concerns the vector $h(\tau) = 
(h_i(\tau))_{0 \le i \le n}$, where $h_i(\tau)$ 
is the number of points in $\qQ_\tau \cap \NN^n$ 
which have exactly $i$ nonzero coordinates.
\begin{conjecture}
{\rm (\cite[Conjecture~0.1]{Cha25+})}
\label{conj:poly}
The vector $h(\tau)$ is equal to the $h$-vector of
an $n$-dimensional simplicial polytope for every
arbor $\tau$ of size $n$. In particular, $h(\tau)$
is palindromic and unimodal.
\end{conjecture}

The palindromicity of $h(\tau)$ (meaning, the 
statement that $h_i(\tau) = h_{n-i}(\tau)$ for 
$0 \le i \le n$) is by itself nontrivial. It 
would imply that the generating polynomial 
$h(\tau,t) = \sum_{i=0}^n h_i(\tau) t^i$ of 
$h(\tau)$ can be written uniquely in the form 
\begin{equation} \label{eq:tau-gamma}
h(\tau,t) = \sum_{i=0}^{\lfloor n/2 \rfloor}
\gamma_i(\tau) t^i (1+t)^{n-2i}
\end{equation}
for some integers $\gamma_0(\tau), \gamma_1
(\tau)\dots,\gamma_{\lfloor n/2 \rfloor}(\tau)$.
The nonnegativity of these numbers, known as 
\emph{$\gamma$-positivity} (see~\cite{Ath18} 
for a survey of this important notion), is 
stronger than the palindromicity and unimodality 
of $h(\tau,t)$. The $\gamma$-positivity of 
$h(\tau,t)$ was implicitly conjectured in 
\cite[Section~0.1]{Cha25+} as well. 

The second main conjecture of~\cite{Cha25+}
concerns the Ehrhart polynomials of arbor 
polytopes. We recall from 
\cite[Section~4.6]{StaEC1} (see also~\cite{FH24}
and references therein) that the \emph{Ehrhart 
polynomial} of a lattice polytope $\qQ \subset 
\RR_{\ge 0}^n$ is the unique polynomial 
$\Ehr(\qQ, t)$ with the property that 
$\Ehr(\qQ, m) = \# (m\qQ \cap \NN^n)$ for every 
$m \in \NN$. The \emph{$h^\ast$-polynomial} 
$h^\ast(\qQ, t)$ of $\qQ$ is defined by the 
equation
\begin{equation} \label{eq:hstar-def} 
\sum_{m \ge 0} \Ehr(\qQ, m) t^m = 
\frac{h^\ast(\qQ, t)}{(1-t)^{d+1}}, 
\end{equation}
where $d = \dim(\qQ)$. The $h^\ast$-polynomial 
always has nonnegative coefficients, a property 
which may fail for the Ehrhart polynomial; see,
for instance, \cite[Section~4]{FH24}. An old 
theorem of Brenti~\cite[Theorem~4.4.4]{Bre89} 
implies that if all roots of $\Ehr(\qQ, t)$ are 
real and lie in the interval $[-1,0]$, then all 
roots of $h^\ast(\qQ, t)$ are real as well. 
\begin{conjecture}
{\rm (\cite[Conjecture~2.2]{Cha25+})}
\label{conj:roots}
For every arbor $\tau$, all roots of the
Ehrhart polynomial of $\qQ_\tau$ are real and lie
in the interval $[-1,0]$. In particular, the
$h^\ast$-polynomial of $\qQ_\tau$ has only real
roots.
\end{conjecture}

The real-rootedness of $h^\ast(\qQ_\tau, t)$ is
implied by a weaker property of the Ehrhart 
polynomial. One may express this polynomial as
\begin{equation} \label{eq:magic-tau}
n! \, \Ehr(\qQ_\tau, t) = \sum_{i=0}^n c_i(\tau)
     t^i(1+t)^{n-i}
\end{equation}
for some integers $c_0(\tau),
c_1(\tau),\dots,c_n(\tau)$. Following~\cite{FH24}, 
we call $\Ehr(\qQ_\tau, t)$ \emph{magic positive}
if these integers are nonnegative. A theorem of
Br\"and\'en \cite[Theorem~4.2]{Bra06} implies in 
this situation that $h^\ast(\qQ_\tau, t)$ has 
only real roots. Thus, one may attempt to prove 
the real-rootedness of the $h^\ast$-polynomial of 
$\qQ_\tau$ by providing a combinatorial 
interpretation of these integers.

Conjecture~\ref{conj:poly} and the 
$\gamma$-positivity of $h(\tau,t)$ were 
verified in the special case of linear arbors 
in~\cite{Ath26}. The aim of this paper is to 
address these questions for the polytopes 
$\qQ_{n,k}$. This polytope is simply the arbor
polytope associated to the arbor $\tau_{n,k}$ 
which has the $(n-k)$-element set 
$\{k+1, k+2,\dots,n\}$ as its root and the 
singletons $\{i\}$ for $i \in [k]$ as the leaves 
of the root (one may call such an arbor an 
\emph{octopus}). Here are two of our main 
results for $\qQ_{n,k}$.
\begin{theorem} \label{thm:gamma}
The polynomial $h(\tau_{n,k},t)$ is
$\gamma$-positive and, in particular, palindromic
and unimodal, for all $n,k$.
\end{theorem}
\begin{theorem} \label{thm:magic}
We have
\begin{equation} \label{eq:magic-nk}
n! \, \Ehr(\qQ_{n,k}, t) = \sum_{i=0}^n c_i t^i
(1+t)^{n-i}
\end{equation}
for some nonnegative integers $c_0,
c_1,\dots,c_n$. In particular, the
$h^\ast$-polynomial of $\qQ_{n,k}$ has only real
roots.
\end{theorem}

Our third main result for $\qQ_{n,k}$ is an 
explicit combinatorial interpretation of its 
$h^\ast$-polynomial. The sum of the coefficients
of this polynomial is known to equal the 
normalized volume of $\qQ_{n,k}$ (defined as the
volume of $\qQ_{n,k}$ multiplied by $n!$). One 
can easily verify (for example, as a 
consequence of Equation~(\ref{eq:Ehr-nk}) in 
Section~\ref{sec:magic}) that the normalized 
volume of $\qQ_{n,k}$ is equal to the number of
maps $f: [n] \to [n]$ which contain the set $[k]$ 
in their image (this is a reason to expect apriori
that $\qQ_{n,k}$ may have interesting lattice point 
enumerative properties). The following statement 
is a natural refinement of this observation. We
recall that $\des(w)$ denotes the number of
descents of a word $w = (w_1, w_2,\dots,w_n) \in 
[n]^n$, meaning indices $i \in [n-1]$ such that 
$w_i > w_{i+1}$.
\begin{theorem} \label{thm:h-star}
We have
\begin{equation} \label{eq:h-star}
h^\ast(\qQ_{n,k}, t) = \sum_{w \in \wW_{n,k}}
       t^{n-1-\des(w)},
\end{equation}
where $\wW_{n,k}$ is the set of words $w \in
[n]^n$ in which each of $1, 2,\dots,k$ appears
at least once. 

In particular, the right-hand 
side of~(\ref{eq:h-star}) has only real roots 
for all $n,k$.
\end{theorem}

Theorem~\ref{thm:gamma} is proven in 
Section~\ref{sec:gamma} by means of an explicit 
combinatorial formula (see 
Proposition~\ref{prop:gamma}) for the 
coefficients $\gamma_i(\tau_{n,k})$ of 
Equation~(\ref{eq:tau-gamma}). 
Theorems~\ref{thm:magic} and~\ref{thm:h-star} 
are proven in Sections~\ref{sec:magic} 
and~\ref{sec:h-star}, respectively. To prove 
Theorem~\ref{thm:magic}, a combinatorial 
interpretation of the coefficients $c_i$ of 
Equation~(\ref{eq:magic-nk}) is provided in 
terms of a new parking model for cars. The 
motivation behind this proof comes 
from~\cite{AFM25}, where a similar result is 
proven for Pitman-Stanley 
polytopes with positive integer parameters. 
Some generalizations, including a conjectural 
generalization of Theorem~\ref{thm:h-star} to 
all arbor polytopes, are discussed in 
Section~\ref{sec:gen}.

Throughout this paper, we use the notation $\NN 
= \{0, 1, 2,\dots\}$ and $[n] = \{1, 2,\dots,n\}$ 
for $n \in \NN$ and denote by $|S|$ the 
cardinality of a finite set $S$. Background on 
Ehrhart theory can be found in~\cite{BR15, FH24,
HiAC} \cite[Chapter~4]{StaEC1}. 

\section{Gamma-positivity}
\label{sec:gamma}

This section proves Theorem~\ref{thm:gamma}. 
That result is a consequence of the explicit 
formula of Proposition~\ref{prop:gamma}. The 
proof is rather computational; we leave the task
to find a bijective proof as an interesting 
open problem. 

The following lemma is a direct consequence 
of the definition of $h(\tau_{n,k},t)$.
\begin{lemma} \label{lem:h}
For every positive integer $n$ and every $k 
\in \{0, 1,\dots,n\}$,
\[ t^n h(\tau_{n,k},1/t) = \sum_{j=0}^n 
   \sum_{i=0}^{n-k} {n-k \choose i}
	 {k \choose j-i}{n-k+j-i \choose j} t^j. \]
\end{lemma}

\begin{proof}
The lattice points in $\qQ_{n,k}$ are the vectors 
$x = (x_1, x_2,\dots,x_n) \in \NN^n$ which satisfy 
\eqref{eq:defQnk}. Given $j \in \{0, 1,\dots,n\}$,
we need to show that the number of these vectors 
which have exactly $j$ zero coordinates is equal 
to 
\[ \sum_{i=0}^{n-k}{n-k \choose i}{k \choose j-i}
   {n-k+j-i \choose j}. \]
To enumerate these elements, let $i$ be the 
number of zero coordinates among $x_{k+1},\dots,x_n$. 
There are ${n-k \choose i}$ ways to choose these 
coordinates and ${k \choose j-i}$ ways to choose 
the remaining zero coordinates of $x$. The $k-j+i$
nonzero coordinates among $x_1, x_2,\dots,x_k$ must 
all be equal to 1. The remaining $n-k-i$ nonzero 
coordinates of $x$ have sum no larger than $n-k+j
-i$. Thus, they can be chosen in ${n-k+j-i \choose 
j}$ ways and the proof follows. 
\end{proof}

\begin{proposition} \label{prop:gamma}
We have
\[ h(\tau_{n,k}, t) = 
   \sum_{i=0}^{\lfloor n/2 \rfloor}{n-k \choose i} 
	 {n-i \choose i} t^i(1+t)^{n-2i} \]
for all $n,k$. In particular, $h(\tau_{n,k}, t)$ is
$\gamma$-positive, and hence palindromic and unimodal.
\end{proposition}

\begin{proof}
By Lemma~\ref{lem:h} we have
\[ [t^{n-j}] h(\tau_{n,k}, t) = \sum_{i=0}^{n-k}
   {n-k \choose i}{k \choose j-i}
	 {n-k+j-i \choose j}. \]
By Vandermonde's identity,
\[ {n-k+j-i \choose j} = \sum_{r=0}^{n-k-i}
   {n-k-i \choose r}{j \choose r} \]
and hence 
\begin{align*}
[t^{n-j}] h(\tau_{n,k},t) &=
\sum_{i=0}^{n-k}{n-k \choose i}{k \choose j-i}
\sum_{r=0}^{n-k-i}{n-k-i \choose r}{j \choose r} 
\\ &= \sum_{r=0}^{n-k}{j \choose r} 
      \sum_{i=0}^{n-k-r}
{n-k \choose i}{k \choose j-i}{n-k-i \choose r}. 
\end{align*}
We now apply the identity
\[ {n-k \choose i}{n-k-i \choose r} =
   {n-k \choose r}{n-k-r \choose i} \]
to obtain that
\begin{align*}
[t^{n-j}] h(\tau_{n,k},t) &=
\sum_{r=0}^{n-k}{n-k \choose r}{j \choose r}
\sum_{i=0}^{n-k-r}{n-k-r \choose i}
{k \choose j-i}\\
&=\sum_{r=0}^{n-k}{n-k \choose r}{j \choose r}
{n-r \choose j},
\end{align*}
where the last equality follows from 
Vandermonde's identity. Since
\[ {j \choose r}{n-r \choose j} = {n-r \choose r}
   {n-2r \choose j-r}, \]
we conclude that
\[ [t^{n-j}] h(\tau_{n,k},t) = \sum_{r=0}^{n-k}
   {n-k \choose r}{n-r \choose r}
	 {n-2r \choose j-r}. \]
Thus, we have shown that
\begin{align*}
t^n h(\tau_{n,k},1/t) &= \sum_{j=0}^{n}
\sum_{r=0}^{n-k}{n-k \choose r}{n-r \choose r}
{n-2r \choose j-r} t^j.
\end{align*}
Finally, we interchange the order of summation 
to conclude that
\begin{align*}
t^n h(\tau_{n,k},1/t) &= \sum_{r=0}^{n-k}
{n-k \choose r}{n-r \choose r}t^r
\sum_{j=0}^{n}{n-2r \choose j-r}t^{j-r}\\
&=\sum_{r=0}^{\lfloor n/2 \rfloor}
{n-k \choose r} {n-r \choose r} t^r(1+t)^{n-2r}.
\end{align*}
This expression is equivalent to the proposed 
identity.
\end{proof}

\begin{remark} \rm
The polynomial $h(\tau_{n,k},t)$ is in fact 
real-rooted. Indeed, by \cite[Remark~3.1.1]{Ga05}, 
this statement is equivalent to the real-rootedness
of
\[ \gamma (\tau_{n,k},t) := 
   \sum_{i=0}^{\lfloor n/2 \rfloor}{n-k \choose i} 
	 {n-i \choose i} t^i. \]
Clearly,
\[ p_{n,k}(t) := \sum_{i=0}^{n-k}{n-k \choose i}
                 t^i = (1+t)^{n-k} \]
is real-rooted and
\[ q_n(t) := \sum_{i=0}^{\lfloor n/2 \rfloor}
   {n-i \choose i}t^i \]
satisfies the recurrence relation $q_n(t) = 
q_{n-1}(t) + t\,q_{n-2}(t)$ and hence is real-rooted
by \cite[Theorem~1.1]{LW07}. Therefore, $\gamma 
(\tau_{n,k},t)$ is real-rooted as the Hadamard 
product of $p_{n,k}(t)$ and $q_n(t)$; see 
\cite[V~155]{PS72} \cite{Wag92} for Hadamard 
products of real-rooted polynomials.
\qed
\end{remark}

\section{Magic positivity}
\label{sec:magic}

This section proves Theorem~\ref{thm:magic}.
The formulas in the next statement follow from 
the relevant definitions in a relatively 
straightforward way.
\begin{proposition} \label{prop:Ehr}
For all $n,k$:
\begin{itemize}
\itemsep=0pt
\item[(a)]
\[ n! \, \Ehr(\qQ_{n,k}, t) = \sum_{j=0}^k
   (-1)^j {k \choose j} \prod_{i=0}^{n-1}
   ((n-j)t + n-j-i). \]

\item[(b)]
\[ n! \, \Ehr(\qQ_{n,k}, t) = \sum_{i=0}^n c_i
   t^i(1+t)^{n-i}, \]
where
\begin{equation} \label{eq:def-ci}
\sum_{j=0}^k (-1)^j {k \choose j}
\prod_{i=0}^{n-1} (it + n-j-i) = \sum_{i=0}^n
c_i t^i. 
\end{equation}
\end{itemize}
\end{proposition}

\begin{proof}
According to the definition of $\qQ_{n,k}$, 
$\Ehr(\qQ_{n,k}, m)$ is equal to the number of 
points $(x_1, x_2,\dots,x_n) \in \NN^n$ such 
that 
\[ \begin{tabular}{l}
$x_1 + x_2 + \cdots + x_n \le mn$, and \\ [1mm] 
$x_i \le m$ \textrm{for} $1 \le i \le k$
\end{tabular} \]
for every $m \in \NN$. Since the number of points 
$(x_1, x_2,\dots,x_n) \in \NN^n$ which satisfy 
the first inequality and violate the second for 
any $j$ given values of $i \in [k]$ is equal to 
$\binom{(n-j)(m+1)}{n}$, it follows by 
inclusion-exclusion that 
\begin{equation} \label{eq:Ehr-nk}
\Ehr(\qQ_{n,k},m) = \sum_{i=0}^{k} (-1)^j 
\binom{k}{j} \binom{(n-j)(m+1)}{n}
\end{equation} 
for every $m \in \NN$. This proves part (a).
Part (b) follows by rewriting the formula of 
part (a) in the form 
\[ n! \, \Ehr(\qQ_{n,k}, t) = \sum_{j=0}^k
   (-1)^j {k \choose j} \prod_{i=0}^{n-1}
   (it + (n-j-i)(1+t)). \]
\end{proof}

{\scriptsize
\begin{table}[hptb]
\begin{center}
\begin{tabular}{| l || l | l | l | l | l |} 
\hline
& $k=0$ & $k=1$ & $k=2$ & $k=3$ & $k=4$ \\ \hline 
\hline
$n=1$   & 1 & 1 & & & \\ \hline
$n=2$  & $2t+2$ & $t+2$ & 2 & & \\ \hline
$n=3$  & $6t^2+15t+6$ & $2t^2+11t+6$ & $6t+6$ & 
          6 & \\ \hline
$n=4$  & $24t^3+104t^2+104t+24$ & 
         $6t^3+59t^2+86t+24$ 
       & $22t^2+64t+24$ & $36t+24$ & 24
				 \\ \hline
\end{tabular}

\bigskip
\caption{The polynomials $f_{n,k}(t)$ for 
$n \le 4$.}
\label{tab:fnk}
\end{center}
\end{table}}

We need to show that the left-hand side of
Equation~(\ref{eq:def-ci}) has nonnegative 
coefficients; see Table~\ref{tab:fnk} for 
its value for $n \le 4$. For that reason, we 
introduce the following variant of the classical
parking protocol for cars (see, for instance,
\cite{Ya15}). Suppose that $n$ cars $C_1, 
C_2,\dots,C_n$ want to park on a street with 
$n$ parking spaces $1, 2,\dots,n$. Any word 
$w = (w_1, w_2,\dots,w_n) \in [n]^n$ may be 
interpreted as a sequence of parking 
preferences, where $w_i$ is the preferred 
parking space of $C_i$. The cars park one at 
a time in the order $C_1, C_2,\dots,C_n$. 
Each car $C_i$ parks in its preferred space 
$w_i$, if available. Otherwise, it parks in
the largest available parking space, i.e., 
it parks in space $n$, if available, 
otherwise in space $n-1$, if available, and 
so on. Note that every car parks, under 
this protocol.

Given $w \in [n]^n$, we say that a car is 
\emph{lucky} if it parks in its preferred 
parking space and denote by $\unlucky(w)$ 
the number of cars which are \emph{not} 
lucky. For example, if $n = 5$ and $w = (3, 
5, 3, 4, 1)$, then $C_1, C_2$ and $C_5$ are 
lucky and $\unlucky(w) = 2$.
\begin{proposition} \label{prop:fnk}
Let
\begin{equation} \label{eq:fnk-def}
f_{n,k}(t) := \sum_{j=0}^k (-1)^j {k \choose j}
\prod_{i=0}^{n-1} (it + n-j-i).
\end{equation}
Then,
\begin{equation} \label{eq:fnk-int}
f_{n,k}(t) = \sum_{w \in \wW_{n,k}}
t^{\unlucky(w)},
\end{equation}
where $\wW_{n,k}$ is the set of words $w \in
[n]^n$ in which each of $1, 2,\dots,k$ appears
at least once.
\end{proposition}

\begin{proof}
We first observe that
\[ f_{n,k}(1+t) = \sum_{j=0}^k (-1)^j {k \choose j}
\prod_{i=0}^{n-1} (it + n-j) \]
and attempt to interpret combinatorially the
right-hand side. The $j=0$ term of the sum may be
interpreted as
\[ \prod_{i=0}^{n-1} (it + n) = \sum_{\varphi \in
\fF_n} t^{\exc(\varphi)}, \]
where $\fF_n$ is the set of maps $\varphi: [n] \to
[2n-1]$ such that $\varphi(i) \le n+i-1$ for every
$i \in [n]$ and $\exc(\varphi)$ is the number of
$i \in [n]$ such that $\varphi(i) > n$. More
generally, for every $j \in \{0, 1,\dots,k\}$ and
all $1 \le i_1 < \cdots < i_j \le k$ we have
\[ \prod_{i=0}^{n-1} (it + n-j) = \sum_{\varphi}
t^{\exc(\varphi)}, \]
where the sum ranges over all maps $\varphi \in
\fF_n$ which contain none of $i_1,\dots,i_j$ in
their image. Thus, by inclusion-exclusion,
\[ f_{n,k}(1+t) = \sum_{w \in \fF_{n,k}}
  t^{\exc(\varphi)}, \]
where $\fF_{n,k}$ stands for the set of maps
$\varphi \in \fF_n$ which contain each of $1,
2,\dots,k$ in their image.

Let $g_{n,k}(t)$ denote the right-hand
side of Equation~(\ref{eq:fnk-int}). To prove
that $f_{n,k}(t) = g_{n,k}(t)$, we will
enumerate the elements of $\fF_{n,k}$ and
$\wW_{n,k}$ with a given set $S$ of places where
$1, 2,\dots,k$ appear for the first time. By
symmetry, we may consider only elements of
these sets for which $1, 2,\dots,k$ appear for
the first time in this order. So, let $S =
\{ s_1, s_2,\dots,s_k\}$ be a $k$-element
subset of $[n]$, where $s_1 < s_2 < \cdots <
s_k$, and let

\begin{align*}
f_{n,k}(S; 1+t) &= \sum_{w \in \fF_{n,k}(S)}
     t^{\exc(\varphi)} \\
g_{n,k}(S; t) &=   \sum_{w \in \wW_{n,k}(S)}
     t^{\unlucky(w)},
\end{align*}
where $\fF_{n,k}(S)$ stands for the set of maps
$\varphi \in \fF_{n,k}$ for which $s_i$ is the
smallest $j \in [n]$ such that $\varphi(j) = i$,
for every $i \in [k]$, and similarly, $\wW_{n,k}
(S)$ stands for the set of words $w \in \wW_{n,k}$
for which $s_i$ is the smallest $j \in [n]$ such
that $w_j = i$, for every $i \in [k]$. Then,

\begin{align*}
f_{n,k}(1+t) &= k! \sum_S f_{n,k}(S; 1+t) \\
g_{n,k}(t) &= k! \sum_S g_{n,k}(S; t),
\end{align*}
where the sums range over all $k$-element subsets
of $[n]$. Hence, it suffices to show that
\begin{equation} \label{eq:f=g}
f_{n,k} (S; 1+t) = g_{n,k}(S; 1+t)
\end{equation}
for every $S$. This will follow from product
formulas for $f_{n,k} (S; 1+t)$ and $g_{n,k}(S;
1+t)$, which are consequences of the
multiplication principle. Indeed, let $S
\subseteq [n]$ be as before and set $s_0 := 0$
and $s_{k+1} := n+1$. Given $\varphi \in
\fF_{n,k}(S)$ and $j \in [n]$ such that $s_i <
j < s_{i+1}$ for some $i \in \{0, 1,\dots,k\}$,
there are $j-1$ possible values for $\varphi(j)$
larger than $n$ and $n-k+i$ possible values $\le
n$, namely the elements of $[n]$ other than $i+1,
i+2,\dots,k$. Hence
\begin{equation} \label{eq:fnk-p}
f_{n,k} (S; 1+t) = \prod_{j \in [n] \sm S}
    p_j(S; t),
\end{equation}
where
\begin{equation} \label{eq:p}
p_j(S; t) = n-k+i + (j-1)t
\end{equation}
for $j \in [n] \sm S$ with $s_i < j < s_{i+1}$.

Similarly, let $w \in \wW_{n,k}(S)$ be a word and
$j \in [n] \sm S$, so that $s_i < j < s_{i+1}$
for some $i \in \{0, 1,\dots,k\}$. We observe
that all cars which correspond to elements of $S$ 
are lucky. Thus, when car $C_j$ attempts to park, 
there are $j-1$ values
of $w_j$ which make car $C_j$ unlucky, namely the
numbers of the spaces where cars $C_1, 
C_2,\dots,C_{j-1}$ have parked, and $n-k-(j-1-i)$ 
values of $w_j$ which make car $j$ lucky, as many 
as the spaces among those numbered as 
$k+1, k+2,\dots,n$ which are free at that stage. 
We conclude that
\begin{equation} \label{eq:gnk-q}
g_{n,k} (S; t) = \prod_{j \in [n] \sm S}
  q_j(S; t),
\end{equation}
where
\begin{equation} \label{eq:q}
q_j(S; t) = n-k+i-j+1 + (j-1)t
\end{equation}
for $j \in [n] \sm S$ with $s_i < j < s_{i+1}$.
We now observe, in view of Equations~(\ref{eq:p})
and~(\ref{eq:q}), that $p_j(S; t) = q_j(S; 1+t)$
for every $j \in [n] \sm S$. Thus, (\ref{eq:f=g})
holds by Equations~(\ref{eq:fnk-p})
and~(\ref{eq:gnk-q}) and the proof follows.
\end{proof}

\begin{proof}[Proof of Theorem~\ref{thm:magic}.]
This follows by combining part (b) of
Proposition~\ref{prop:Ehr} with
Proposition~\ref{prop:fnk}.
\end{proof}

\section{The $h^*$-polynomial}
\label{sec:h-star}

This section proves Theorem~\ref{thm:h-star}.
\begin{proof}[Proof of Theorem~\ref{thm:h-star}.]
By Equation~(\ref{eq:Ehr-nk}), 
\[ \begin{aligned}
   \Ehr(\qQ_{n,k},m) &= \sum_{i=0}^k (-1)^i 
        \binom{k}{i} \binom{(n-i)(m+1)}{n} \\
&= \sum_{i=0}^k (-1)^i \binom{k}{i} [t^n] 
                (1+t)^{(n-i)(m+1)} \\
&= [t^n] \sum_{i=0}^k (-1)^i \binom{k}{i} 
                (1+t)^{(n-i)(m+1)} \\
&= [t^n] (1+t)^{n(m+1)} \sum_{i=0}^k (-1)^i 
    \binom{k}{i} \bigl( (1+t)^{-(m+1)} \bigr)^i \\
&= [t^n] (1+t)^{n(m+1)} \bigl( 1 - (1+t)^{-(m+1)} 
          \bigr)^k \\ \bigskip
&= [t^n] (1+t)^{(n-k)(m+1)} 
          \bigl( (1+t)^{m+1} - 1 \bigr)^k.
\end{aligned} \]
We expand the two factors as
\[ \bigl( (1+t)^{m+1} - 1 \bigr)^k
   = \prod_{i=1}^k \Bigl( \sum_{a_i = 1}^{m+1} 
	 \binom{m+1}{a_i} t^{a_i} \Bigr), \]
\[ (1+t)^{(n-k)(m+1)} = \prod_{i=k+1}^n \Bigl( 
   \sum_{a_i = 0}^{m+1} \binom{m+1}{a_i} t^{a_i} 
	 \Bigr) \]
and extract the coefficient of $t^n$ to conclude
that
\begin{equation}\label{eq:Ehr-wnk}
\Ehr(\qQ_{n,k},m) =
\sum_{\substack{a_1 + \dots + a_n = n \\
a_1,\dots,a_k \ge 1,\; a_{k+1},\dots,a_n \ge 0}}
\; \prod_{i=1}^{n} \binom{m+1}{a_i},
\end{equation}
where the sum ranges over all weak compositions 
$\mu = (a_1, a_2,\dots,a_n)$ of $n$ such that 
$a_1, a_2,\dots,a_k \ge 1$.  For such a weak 
composition $\mu$, we denote by $\sS_\mu$ the 
set of all permutations of the multiset
containing $a_i$ copies of $i$ for $1 \le i \le 
n$. MacMahon's formula \cite[p.~211]{Mac04} 
states that 
\[ \sum_{m \ge 0} \binom{m+a_1}{a_1} 
   \binom{m+a_2}{a_2} \!\cdots\!
   \binom{m+a_n}{a_n} t^m
   = \frac{\displaystyle\sum_{w \in \sS_\mu} 
	 t^{\des(w)}}{(1-t)^{\,n+1}}. \]
Since $a_i>0$ for at least one $i$, we may 
replace $m$ by $-m$ in the coefficient of $t^m$ 
on the left-hand side and apply 
\cite[Proposition~4.2.3]{StaEC1} to reach the 
identity
\[ \sum_{m\ge 0} \binom{m+1}{a_1} 
   \binom{m+1}{a_2} \!\cdots\! \binom{m+1}{a_n} 
	 t^m = \frac{\displaystyle\sum_{w \in \sS_\mu} 
   t^{\,n-1-\des(w)}}{(1-t)^{\,n+1}}. \]
We now sum this identity over all weak 
compositions $\mu$ of $n$ satisfying the 
conditions in~(\ref{eq:Ehr-wnk}) and observe 
that the union $\bigcup_\mu {\sS_\mu}$, where 
$\mu$ ranges over these weak compositions, is 
exactly the set $\wW_{n,k}$. We conclude that
\[ \sum_{m\ge 0} \Ehr(\qQ_{n,k},m) \,t^m
 = \frac{\displaystyle\sum_{w \in \wW_{n,k}} 
 t^{\,n-1-\des(w)}}{(1-t)^{n+1}} \]
and the proof follows.
\end{proof}

\section{Generalizations}
\label{sec:gen}

A natural generalization of $\qQ_{n,k}$ is the 
lattice polytope $\qQ_{n,d,k}$ in $\RR^n$ 
defined by the inequalities
\[ \begin{tabular}{l}
$x_i \ge 0$ \textrm{for} $1 \le i \le n$, \\
[1mm] $x_i \le 1$ \textrm{for} $1 \le i \le k$,
and \\ [1mm] $x_1 + x_2 + \cdots + x_n \le d$,
\end{tabular} \]
where $n, d$ are positive integers and $k \in 
\NN$ is such that $k \le \min \{n, d\}$. A nice 
feature of this polytope is that its normalized
volume is equal to the number of maps $f: [n] 
\to [d]$ which contain $1, 2,\dots,k$ in their 
image. This follows by computing the coefficient
of $t^n/n!$ in the expression 
\begin{equation} \label{eq:Ehr-ndk}
\Ehr(\qQ_{n,d,k},t) = \sum_{i=0}^{k} (-1)^j 
\binom{k}{j} \binom{(d-j)t+n-j}{n},
\end{equation} 
which directly generalizes 
Equation~(\ref{eq:Ehr-nk}).

One can easily find examples of polytopes 
$\qQ_{n,d,k}$ with $d < n$ (such as the 
3-dimensional simplex $\qQ_{3,2,0}$) for which 
the Ehrhart polynomial is \emph{not} magic 
positive. The following generalization of 
Theorem~\ref{thm:magic} gives a positive 
result for $d \ge n$. We sketch the proof,
which directly extends that of 
Theorem~\ref{thm:magic}. The main difference 
is in the parking protocol used, where parking
preferences now are words $w = (w_1, 
w_2,\dots,w_n) \in [d]^n$ (there are again $n$
cars $C_1, C_2,\dots,C_n$ and $n$ parking 
spaces $1, 2,\dots,n$). Every preference $w_i 
> n$ is considered unavailable and thus, car 
$C_i$ parks in the largest available space.
\begin{theorem} \label{thm:magic-gen}
The Ehrhart polynomial of $\qQ_{n,d,k}$ is magic
positive whenever $d \ge n$. More precisely, we 
have
\[ n! \, \Ehr(\qQ_{n,d,k}, t) = \sum_{i=0}^n c_i 
   t^i(1+t)^{n-i}, \]
where 
\begin{equation} \label{eq:def-g-ndk}
\sum_{i=0}^n c_i t^i = \sum_{w \in \wW_{n,d,k}}
   t^{\unlucky(w)} 
\end{equation}
and $\wW_{n,d,k}$ is the set of words $w \in
[d]^n$ in which each of $1, 2,\dots,k$ appears
at least once.
\end{theorem}

\begin{proof}
We may rewrite Equation~(\ref{eq:Ehr-ndk}) as
\begin{align*} 
n! \, \Ehr(\qQ_{n,d,k}, t) &= 
   \sum_{j=0}^k (-1)^j {k \choose j} 
	 \prod_{i=0}^{n-1} ((d-j)t + n-j-i) \\
&= \sum_{i=0}^n c_i
   t^i(1+t)^{n-i}, 
\end{align*}
where
\begin{equation} \label{eq:def-ci-d}
\sum_{j=0}^k (-1)^j {k \choose j}
\prod_{i=0}^{n-1} ((d-n+i)t + n-j-i) = 
\sum_{i=0}^n c_i t^i. 
\end{equation}
Let us denote the left-hand side of 
Equation~(\ref{eq:def-ci-d}) by $f_{n,d,k}(t)$
and the right-hand side of 
Equation~(\ref{eq:def-g-ndk}) by $g_{n,d,k}(t)$.
We need to show that $f_{n,d,k}(t) = g_{n,d,k}
(t)$. We note that 
\[ f_{n,d,k}(1+t) = \sum_{j=0}^k 
   (-1)^j {k \choose j} 
	 \prod_{i=0}^{n-1} ((d-n+i)t + d-j) \]
and interpret the $j=0$ term of this sum as 
\[ \prod_{i=0}^{n-1} ((d-n+i)t + d) = 
   \sum_{\varphi \in \fF_{n,d,0}} 
	 t^{\exc(\varphi)}, \]
where $\fF_{n,d,0}$ is the set of maps 
$\varphi: [n] \to [2d-1]$ such that $\varphi
(i) \le 2d-n+i-1$ for every $i \in [n]$ and 
$\exc(\varphi)$ is the number of $i \in [n]$ 
such that $\varphi(i) > d$. As in the proof 
of Proposition~\ref{prop:fnk}, we deduce by 
inclusion-exclusion that 
\[ f_{n,d,k}(1+t) = \sum_{w \in \fF_{n,d,k}}
  t^{\exc(\varphi)}, \]
where $\fF_{n,d,k}$ is the set of maps 
$\varphi \in \fF_{n,d,0}$ which contain each 
of $1, 2,\dots,k$ in their image. For $S =
\{ s_1, s_2,\dots,s_k\} \subseteq [n]$ with 
$s_1 < s_2 < \cdots < s_k$, we let

\begin{align*}
f_{n,d,k}(S; 1+t) &= \sum_{w \in \fF_{n,d,k}
   (S)} t^{\exc(\varphi)} \\
g_{n,d,k}(S; t) &= \sum_{w \in \wW_{n,d,k}(S)}
     t^{\unlucky(w)},
\end{align*}
where $\fF_{n,d,k}(S)$ is the set of maps
$\varphi \in \fF_{n,d,k}$ for which $s_i$ is 
the smallest $j \in [n]$ such that $\varphi(j) 
= i$, for every $i \in [k]$, and similarly, 
$\wW_{n,d,k}(S)$ is the set of words $w \in 
\wW_{n,d,k}$ for which $s_i$ is the smallest 
$j \in [n]$ such that $w_j = i$, for every 
$i \in [k]$. Then, it suffices to show that
$f_{n,d,k} (S; 1+t) = g_{n,d,k}(S; 1+t)$ for every 
$S$ and the proof of Proposition~\ref{prop:fnk}
works, with the modified formulas
\[ p_j(S; t) = d-k+i + (d-n+j-1)t \]
and
\[ q_j(S; t) = n-k+i-j+1 + (d-n+j-1)t \]
for $j \in [n] \sm S$ with $s_i < j < s_{i+1}$; 
the details are omitted.
\end{proof}

We conclude with a conjectural generalization 
of Theorem~\ref{thm:h-star} for every arbor 
$\tau$. We recall that the set $\dD(v)$ is 
defined in the introduction for all vertices 
$v$ of $\tau$. Extensive experimentation by 
computer supports the following conjecture. 
\begin{conjecture}
\label{conj:h-star}
For every arbor $\tau$ of size $n$ 
\[ h^\ast(\qQ_\tau, t) = \sum_{w \in \wW_\tau}
       t^{n-1-\des(w)}, \]
where $\wW_\tau$ is the set of words $w \in 
[n]^n$ with the following property: for every 
vertex $v$ of $\tau$, the elements of $\dD(v)$ 
appear a total of at least $|\dD(v)|$ times in 
$w$.
\end{conjecture}

The validity of this statement is unclear even
for linear arbors, studied in~\cite{Ath26}. An 
even more general conjecture will appear
in~\cite{AC26+}. 

\medskip
\noindent
\textbf{Acknowledgements}. The second and third
named authors acknowledge support from the China 
Scholarship Council (no. 202306060164 and 
202508370093, respectively) for their visit to 
the Department of Mathematics of the National and
Kapodistrian University of Athens during the 
academic year 2025--26. The authors also wish to 
thank Luis Ferroni for making~\cite{AFM25} 
available to them and Frederic Chapoton for 
helpful discussions.


\begin{thebibliography}{99}
%
\bibitem{Ath18}
C.A.~Athanasiadis,
\textit{Gamma-positivity in combinatorics and geometry},
S\'em. Lothar. Combin. {\bf~77} (2018), Article B77i,
64pp (electronic).
%
\bibitem{Ath26}
C.A.~Athanasiadis,
\textit{Lattice point enumeration of polytopes 
associated
to integer compositions}, Ann. Comb. (to appear),
{\tt arXiv:2510.23903}.
%
\bibitem{AC26+}
C.A.~Athanasiadis and F.~Chapoton,
\textit{Polytopes and posets associated to preorders},
in preparation.
%
\bibitem{AFM25}
N.~Avila, L.~Ferroni and A.H.~Morales,
\textit{Luck and magic for Pitman--Stanley polytopes},
FPSAC Extended Abstract, 2025.
%
\bibitem{BR15}
M.~Beck and S.~Robins,
Computing the Continuous Discretely: Integer-Point
Enumeration in Polyhedra,
Springer, 2015.
%
\bibitem{Bra06}
P.~Br\"and\'en,
\textit{On linear transformations preserving the P\'olya 
frequency property},
Trans. Amer. Math. Soc. {\bf~358} (2006), 3697--3716.
%
\bibitem{Bre89}
F.~Brenti,
\textit{Unimodal, log-concave and P\'olya frequency 
sequences in combinatorics},
Mem. Amer. Math. Soc. {\bf~81} (1989), no. 413, 
pp.~viii+106.
%
\bibitem{Cha25+}
F.~Chapoton,
\textit{On posets and polytopes attached to arbors},
Math. Scand. {\bf~131} (2025), 401--449.
%
\bibitem{FH24}
L.~Ferroni and A.~Higashitani,
\textit{Examples and counterexamples in Ehrhart 
theory},
EMS Surv. Math. Sci. (to appear), 
{\tt arXiv:2307.10852}.
%
\bibitem{Ga05}
S.R.~Gal,
\textit{Real root conjecture fails for five- 
and higher-dimensional spheres},
Discrete Comput. Geom. {\bf~34} (2005), 269--284.
%
\bibitem{HiAC}
T.~Hibi,
Algebraic Combinatorics on Convex Polytopes,
Carslaw Publications, Australia, 1992.
%
\bibitem{LW07}
L.L.~Liu and Y.~Wang, 
\textit{A unified approach to polynomial sequences 
with only real zeros}, 
Adv. in Appl. Math. {\bf~38} (2007), 542--560.
%
\bibitem{Mac04}
P.A.~MacMahon,
Combinatorial analysis, Vol.~II,
Dover Publications, 2004.
%
\bibitem{PS72}
G.~P\'olya and G.~Szego,
Problems and Theorems in Analysis, vol.~2,
Springer--Verlag, 1972.
%
\bibitem{StaEC1}
R.P.~Stanley,
Enumerative Combinatorics, vol.~1,
Cambridge Studies in Advanced Mathematics {\bf~49},
Cambridge University Press, second edition,
Cambridge, 2011.
%
\bibitem{Wag92}
D.G.~Wagner, 
\textit{Total positivity of Hadamard products}, 
J. Math. Anal. Appl. {\bf~163} (1992), 459--483.
%
\bibitem{Ya15}
C.H.~Yan,
\textit{Parking functions},
in \textit{Handbook of Combinatorics} (M.~Bona, ed.), 
CRC Press, 2015, pp.~835--893.
%
\end{thebibliography}
\end{document}